\documentclass[10pt]{amsart}
\usepackage{geometry}                
\usepackage{graphicx}
\usepackage{amssymb}
\usepackage{epstopdf}
\newtheorem{theorem}{Theorem}
\newtheorem{corollary}[theorem]{Corollary}
\newtheorem{lemma}[theorem]{Lemma}
\newtheorem{prop}[theorem]{Propostion}

\newcommand{\Vol}{\operatorname{Vol}}
\newcommand{\orth}{\operatorname{L}}

\title[Lower bounds for volumes and orthospectra]{Lower bounds for volumes and orthospectra of hyperbolic manifolds with geodesic boundary}
\author{Mikhail Belolipetsky}
\thanks{Belolipetsky is partially supported by CNPq, FAPERJ and MPIM in Bonn.} 
\address{IMPA, Estrada Dona Castorina, 110, 22460-320 Rio de Janeiro, Brazil}
\email{mbel@impa.br}

\author{Martin Bridgeman}
\thanks{Bridgeman is supported by NSF grant DMS-2005498 and by a grant from the Simons Foundation (675497, MJB)}
\address{Boston College, Chestnut Hill, MA 02467, USA}
\email{bridgem@bc.edu}

\begin{document}
\begin{abstract}
In this paper we derive explicit estimates for the functions which appear in the previous work of Bridgeman and Kahn. As a consequence, we obtain an explicit lower bound for the length of the shortest orthogeodesic in terms of the volume of a hyperbolic manifold with totally geodesic boundary. We also give an alternative derivation of a lower bound for the volumes of these  manifolds as a function of the dimension. 
\end{abstract}

\maketitle

 \section{Introduction}
Let $M$ be a compact hyperbolic $n$-dimensional manifold with non-empty totally geodesic boundary. An {\em orthogeodesic} of $M$ is a geodesic arc  with endpoints in $\partial M$ which are perpendicular to $\partial M$ at the endpoints. The {\em orthospectrum}  $\Lambda_M$  of $M$ is the set (with multiplicities) of lengths of orthogeodesics. As the orthogeodesics of $M$ correspond to a subset of the closed geodesics of its double, the set of orthogeodesics of $M$ is countable. We let $\Vol(M)$ and $\Vol(\partial M)$ be the volumes of the hyperbolic manifolds $M$ and $\partial M$. We further let $\orth(M)$ be the length of the shortest orthogeodesic of $M$. In this paper we will explore the relation between the three quantities $\Vol(M), \Vol(\partial M)$ and $\orth(M)$.

The orthospectrum was first introduced by Basmajian in the 1993 paper ``The orthogonal spectrum of a hyperbolic manifold'' (see  \cite{Bas93}). In the paper Basmajian showed that a totally geodesic hypersurface $S$ in a hyperbolic manifold can be decomposed into embedded disks which are in one-to-one correspondence with the orthogeodesics of the manifold $M$ obtained by cutting  along the hypersurface $S$. Then, by describing the radii of the disks in terms of the length of the corresponding orthogeodesics, Basmajian obtained the following orthospectrum identity
$$ \Vol(S) = \sum_{l \in \Lambda_{M}} V_{n-1}\left(\log\left(\coth{\frac{l}{2}}\right)\right),$$
where  $V_n(r)$ is the volume of a hyperbolic ball of radius $r$ in ${\bf H}^n$.

Using a decomposition of the tangent bundle via orthogeodesics, the second author and Kahn proved the following.

\begin{theorem}{(Bridgeman--Kahn, \cite{BK10})}
Given $n \geq 2$ there exists a continuous monotonically decreasing function $F_n:{\bf R}_+ \rightarrow {\bf R}_+$ such that if $M$ is a compact hyperbolic $n$-manifold with non-empty totally geodesic boundary, then
$$\Vol(M) = \sum_{l \in \Lambda_M} F_n(l).$$
\label{BKid}
\end{theorem}

The function $F_n$ is given by an integral formula, see equation \eqref{Fn} below.  The above theorem was generalized to non-compact finite volume hyperbolic manifolds with totally geodesic boundary by Vlamis and Yarmola (see \cite{VY}).

An analysis of the  asymptotic behaviour of $F_n(l)$  as $l \rightarrow 0$ gives

\begin{theorem}{(Bridgeman--Kahn, \cite{BK10})}
For $n \geq 3$, there exists a monotonically increasing function $H_n:{\bf R}_+ \rightarrow{\bf R}_+$ and a constant $C_n > 0$ such that if $M$ is a compact hyperbolic $n$-manifold with totally geodesic boundary with $\Vol(\partial M) = A$, then 
$$\Vol(M) \geq H_n(A) \geq C_n \cdot A^{\frac{n-2}{n-1}}.$$
\label{vol}
\end{theorem}

The functions $F_n$, $H_n$, and the implied constants $C_n$ which appear in \cite{BK10} are defined by complicated formulas and it is difficult to evaluate or estimate them. In this paper we resolve this issue and find explicit lower bounds in terms of the dimension $n$.  We first prove the following relation between $\Vol(M)$ and $\orth(M)$.

\begin{theorem}\label{thm}
For $n \geq 3$, if $M$ is a compact hyperbolic $n$-manifold $M$ with totally geodesic boundary, then either  $\orth(M) \geq \frac{1}{2}\log\left(\frac{5}{2}\right)$ or
\begin{equation}\label{eq1}
e^{\orth(M)}-1 \geq g_n\sqrt{\frac{2\pi e}{n-1}}\cdot (\Vol(M))^{-\frac{1}{n-2}},
\end{equation}
where $g_n$ is an explicit monotonically increasing function tending to $1$. 
\end{theorem}

The function $g_n$ is given by equation \eqref{eq g_n} below. In particular, the first few approximate values are $g_3 = 0.120822$, $g_4 = 0.464543$, $g_5=0.563796$, $g_6 = 0.617183$.

One consequence of Theorem~\ref{thm} is the following dichotomy between volume and shortest orthogeodesic:
\begin{corollary}
Let $M$ be a compact hyperbolic manifold with non-empty totally geodesic boundary of dimension $n \geq 3$. Then  either
$$\mbox{ $\Vol(M) \geq 1$}\qquad\mbox{ or } \qquad e^{\orth(M)}-1 \geq \min\left(\sqrt{\frac{5}{2}}-1,g_n\sqrt{\frac{2\pi e}{n-1}}\right).$$
\end{corollary}

The results of \cite{BK10} have a number of applications that can be made more precise now. For example, in \cite{BT11} they were used to estimate the volumes of hyperbolic manifolds with small systole constructed there. Inequality \eqref{eq1} allows us to restate the inequality from \cite[Theorem~1.2]{BT11}:
\begin{corollary}
Hyperbolic manifolds with small systole constructed by Belolipetsky--Thomson in \cite{BT11} satisfy
$$ \Vol(M) \geq \left(\frac{g_n}{2}\sqrt{\frac{2\pi e}{n-1}} \cdot \frac{1}{{\mathrm{Syst}_1(M)}}\right)^{n-2}. $$
\end{corollary}

We also use our analysis to investigate the relation between $\Vol(M)$ and $\Vol(\partial M)$ which  we compare  with the results of Miyamoto in \cite{Miy94}. We prove

\begin{theorem}\label{thm2} 
Let $M$ be a compact hyperbolic manifold with non-empty totally geodesic boundary of dimension $n \geq 3$. Then  either
\begin{align}\label{eq2}
\Vol(M) & \geq \frac{1}{4}\log\left(\frac{5}{2}\right)\Vol(\partial M) \nonumber\\
\text{ or }\quad &\\
\Vol(M)  &\geq \frac{h_n}{3}\sqrt{\frac{2\pi e}{n-1}}\cdot (\Vol(\partial M))^{\frac{n-2}{n-1}},\nonumber
\end{align}
where $h_n$ is an explicit monotonically increasing function tending to $1$.
\end{theorem}

The function $h_n$ is given by equation \eqref{eq h_n} with the first few approximate values $h_3 = 0.203335$, $h_4 = 0.448875$, $h_5 = 0.542675$, $h_6 = 0.601147$.

In earlier work Miyamoto obtained a lower bound for the volume in terms of a linear function of the volume of the boundary:
\begin{theorem}
{(Miyamoto, \cite[Theorem~4.2]{Miy94})}
Let $M$ be a hyperbolic $n$-manifold with totally geodesic boundary. Then there are constants $\rho_n > 0$ such that
\begin{equation}
\Vol(M) \geq \rho_n \cdot \Vol(\partial M).
\label{meq}
\end{equation}
\end{theorem}

One application of both \eqref{eq2} and  \eqref{meq} is to obtain lower bounds on the volume of a hyperbolic manifold with totally geodesic boundary in terms of the dimension. Although both use very different methods, their resulting bounds are surprisingly similar. 

For $n$ even,  applying the Gauss-Bonnet formula for the double $DM$ gives
$$\Vol(M) = \frac{1}{2}\Vol(DM) = \frac{|\chi(DM)|}{4}V_{n} \geq \frac{1}{4}V_{n},$$
where $V_{n}$ is the volume of the unit $n$-sphere in ${\bf R}^{n+1}$. For $n$ odd, both \eqref{eq2} and \eqref{meq} can be used to leverage the Gauss--Bonnet theorem on the boundary to give lower bounds for the volume of the manifolds. 

In  \cite{KelHab}, Kellerhals used packing estimates to show that Miyamoto's function $\rho_n$ is monotonically increasing with the approximate values $\rho_3 = 0.29156$, $\rho_4 = 0.43219$, $\rho_5 = 0.54167$, $\rho_6 = 0.64652$. 

Thus for $M$ a hyperbolic $n$-manifold with non-empty totally geodesic boundary and $n$ odd we have
 $$\Vol(M) \geq \frac{\rho_n}{2} V_{n-1}.$$

Using our bound in \eqref{eq2} we can derive a similar estimate. We prove
\begin{theorem}\label{thm-vol}
Let $M$ be a hyperbolic $n$-manifold with non-empty totally geodesic boundary and $n$ odd. Then
$$\Vol(M) \geq  \min\left(\frac{1}{8}\log\left(\frac{5}{2}\right),\frac{h_n}{6}\right) V_{n-1}.$$
\end{theorem}

The paper is organized as follows. We first describe the functions $F_n(x)$, $M_n(x)$ and by a careful analysis obtain uniform lower bounds for each as functions of $n$ and $x$. An important step is bounding an incomplete Beta function which requires us to restrict to $x \leq \frac{1}{2}\log\left(\frac{5}{2}\right)$ (see Lemma~\ref{Fb}).  We then apply these bounds to prove the bounds on volume and ortholength  in Theorems~\ref{thm}  and \ref{thm2} above. In Section~\ref{sec5} we consider more carefully the three dimensional case. In Section~\ref{sec6} we conclude with the proof of Theorem~\ref{thm-vol} and a related discussion.

\medskip

{\noindent\bf Acknowledgments.} We thank Ruth Kellerhals for helpful correspondence. We would also like to thank the referee for their comments and insights which improved the paper. 

\section{The functions $F_n$ and $M_n$}
In the prior paper, an integral formula for $F_n$ is derived. We let $V_k$ be the volume of the unit k-sphere in ${\bf R}^{k+1}$. Then from \cite{BK10} we have\footnote{The original formula had an incorrect factor of $2$ rather than $2^{n-1}$ which was  corrected by \cite[Theorem 2.1]{VY}.}
\begin{equation}
F_n(l) =\frac{2^{n-1}V_{n-2}V_{n-3}}{V_{n-1}}\int_0^1 \frac{r^{n-3}}{\left(\sqrt{1-r^2}\right)^{n-2}}.M_n\left(\sqrt{\frac{e^{2l}-r^2}{1-r^2}}\right)dr,
\label{Fn}
\end{equation}
where 
\begin{equation}
M_n(b) = \int^{1}_{-1} du\int^{\infty}_{b}  \frac{\log\left(\frac{(v^2-1)(u^2-b^2)}{(v^2-b^2)(u^2-1)}\right)}{(v-u)^{n}}dv.
\label{Mneqn}
\end{equation}

Furthermore it is shown that the function $M_n(b)$ can be given in terms of standard functions. In order to describe this function, we define the following. For  $n \geq 1$  we define the polynomial function $P_n$ by
$$P_n(x) = \sum_{k=1}^n \frac{x^k}{k}.$$
We also define $P_0(x) = 0$. We note that for $|x| < 1$, $P_n(x)$ is the first $n$ terms of the Taylor series of $-\log(1-x)$. 
We therefore define the function $L_n(x)$ by
$L_n(x) = \log|1-x| + P_n(x).$
For  $|x| < 1$ we have
$$L_n(x) = -\sum_{k=n+1}^\infty \frac{x^k}{k}.$$
We note that $L_0(x) = \log|1-x|$. We also note that $P_n(1) = 1 + \frac{1}{2}+\ldots+\frac{1}{n}$, the n$^{th}$ Harmonic number. Using these functions, $M_n$ can be written down explicitly.

\begin{lemma}{(Bridgeman--Kahn, \cite[Lemma 7]{BK10})}
The function  $M_n:(1,\infty) \rightarrow {\bf R}_+$ has the explicit form
$$(n-1)(n-2)M_n(b) = $$
$$\frac{1}{(b-1)^{n-2}}\left( \log\left(\frac{(b+1)^2}{4b}\right)+2P_{n-2}(1)-L_{n-3}\left(\frac{b-1}{b+1}\right)-(-1)^n L_{n-3}\left(\frac{-b+1}{b+1}\right)\right)$$
$$ +\frac{1}{(b+1)^{n-2}}\left(-\log\left(\frac{(b-1)^2}{4b}\right)-2P_{n-2}(1)+L_{n-3}\left(\frac{b+1}{b-1}\right) +(-1)^n L_{n-3}\left(\frac{-b-1}{b-1}\right)\right)$$
$$+\frac{1}{(2b)^{n-2}}\left(L_{n-3}\left(\frac{2b}{b+1}\right)-L_{n-3}\left(\frac{2b}{b-1}\right)\right)+\frac{1}{2^{n-2}}\left(L_{n-3}\left(\frac{2}{b+1}\right)-(-1)^nL_{n-3}\left(\frac{-2}{b-1}\right)\right).$$
Furthermore, $M_n$ satisfies 
$$\lim_{b \rightarrow 1^+} (b-1)^{n-2}M_n(b) =  \frac{2P_{n-2}(1)}{(n-1)(n-2)} \qquad\mbox{and}\qquad
  \lim_{b \rightarrow \infty} \frac{b^{n-1}}{\log{b}}.M_n(b) = \frac{4}{n-1}.
 $$
\label{Mn}
\end{lemma}
 
We note that the above is a consequence of the following formulae.

\begin{lemma}{(Bridgeman--Kahn, \cite[Corollary 6]{BK10})} For $n \geq 2$
$$\int \frac{\log|x-a|}{(x-b)^n} dx = \frac{1}{n-1}\left(\frac{L_{n-2}\left(\frac{a-b}{x-b}\right)}{(a-b)^{n-1}} -\frac{\log|x-a|}{(x-b)^{n-1}}\right).$$
Furthermore, for $k \geq 1$
$$\lim_{x \rightarrow a} \left(\frac{\log|x-a|}{(b-x)^{k}}-\frac{L_{n}\left(\frac{b-a}{b-x}\right)}{(b-a)^{k}}\right) = 
\frac{\log|b-a|-P_{n}(1)}{(b-a)^{k}}.$$
\label{L_n_formula}\end{lemma}

\section{Explicit lower bounds for $F_n$, $M_n$}
In this section, we give explicit lower bounds on for the functions $F_n, M_n$. As these functions are only defined for $n \geq 3$, in the following a standing assumption is that $n \geq 3$. In order to obtain our bounds, we need to derive a  lower bound on $M_n(b)$ which is uniform  both in $n$ and $b$. By Lemma~\ref{Mn} we have
$$\lim_{b \rightarrow 1^+} (b-1)^{n-2}M_n(b) =  \frac{2P_{n-2}(1)}{(n-1)(n-2)}.$$
We prove the following uniform lower bound. 
\begin{lemma}
$$(b-1)^{n-2}M_n(b) \geq \frac{P_{n-3}(1)+\left(1-\frac{1}{3^{n-2}}\right)\left(P_{n-2}(1)+\log\left(3/4\right)\right)}{(n-1)(n-2)}  \qquad b \in (1,2].$$
\label{Munif}
\end{lemma}

\begin{proof}
From  equation \ref{Mneqn} for $M_n$ we have that
 $$M_n(b)  = \int^{1}_{-1} du\int^{\infty}_{b}  \frac{\log\left(\frac{(v^2-1)(b^2-u^2)}{(v^2-b^2)(1-u^2)}\right)}{(v-u)^{n}}dv \geq \int^{1}_{-1} du\int^{\infty}_{b}  \frac{\log\left(\frac{(v^2-1)(b-u)}{(v^2-b^2)(1-u)}\right)}{(v-u)^{n}}dv$$
 as $b+u > 1+u$.
 We split the interior integral on the right into two integrals.
 $$I_1=-\int^{\infty}_{b}  \frac{\log\left(v-b\right)}{(v-u)^{n}}dv, \qquad I_2 = \int^{\infty}_{b}  \frac{\log\left(\frac{(v^2-1)(b-u)}{(v+b)(1-u)}\right)}{(v-u)^{n}}dv.$$
 By Lemma \ref{L_n_formula} we have
 $$I_1= \left.\frac{1}{n-1}\left(\frac{\log(v-b)}{(v-u)^{n-1}}-\frac{L_{n-2}\left(\frac{b-u}{v-u}\right)}{(b-u)^{n-1}}\right)\right|_b^\infty = \frac{1}{n-1}\lim_{v\rightarrow b^+}\left(\frac{L_{n-2}\left(\frac{b-u}{v-u}\right)}{(b-u)^{n-1}}-  \frac{\log(v-b)}{(v-u)^{n-1}}\right)$$
 By the limit in Lemma \ref{L_n_formula} we have
 
 $$I_1= \frac{1}{n-1}\left(\frac{P_{n-2}(1)-\log(b-u)}{(b-u)^{n-1}}\right).$$
Integrating by parts we get
$$I_2 = \left.-\frac{1}{n-1}\left( \frac{\log\left(\frac{(v^2-1)(b-u)}{(v+b)(1-u)}\right)}{(v-u)^{n-1}}\right|_b^\infty +\int_b^\infty\frac{dv}{(v-u)^{n-1}}\left(\frac{1}{v-1}+\frac{1}{v+1}-\frac{1}{v+b}\right)\right).$$
As $v+b > v+1$ we have
$$\frac{1}{v-1}+\frac{1}{v+1}-\frac{1}{v+b} \geq \frac{1}{v-1} > 0.$$
Therefore
$$I_2 \geq \left.-\frac{1}{n-1}\left( \frac{\log\left(\frac{(v^2-1)(b-u)}{(v+b)(1-u)}\right)}{(v-u)^{n-1}}\right|_b^\infty \right) = \frac{1}{n-1}\left( \frac{\log\left(\frac{(b^2-1)(b-u)}{(2b)(1-u)}\right)}{(b-u)^{n-1}} \right).$$
Therefore combining we have
$$M_n(b) \geq \frac{1}{n-1}\left( \int_{-1}^{1} \frac{\log\left(\frac{(b^2-1)}{2b(1-u)}\right)+P_{n-2}(1)}{(b-u)^{n-1}}du\right) =   J_1(b) + J_2(b),$$
where
$$J_1(b) = \frac{1}{n-1}\left( \int_{-1}^{1} \frac{\log\left(\frac{(b^2-1)}{2b}\right)+P_{n-2}(1)}{(b-u)^{n-1}}du\right),$$
$$J_2(b) =  \frac{1}{n-1}\left(\int_{-1}^{1} \frac{-\log(1-u)}{(b-u)^{n-1}}du\right).$$
By integration we have
$$J_1(b) =   \frac{\log\left(\frac{(b^2-1)}{2b}\right)+P_{n-2}(1)}{(n-1)(n-2)}\left(\frac{1}{(b-1)^{n-2}} - \frac{1}{(b+1)^{n-2}}\right).$$
Using Lemma \ref{L_n_formula} we get
$$J_2(b) = \left.\frac{1}{(n-1)(n-2)}\left(\frac{-\log(1-u)}{(b-u)^{n-2}} 
+\frac{L_{n-3}(\frac{b-1}{b-u})}{(b-1)^{n-2}} \right)\right|_{-1}^{1}.$$
Therefore,
$$(n-1)(n-2)J_2(b) = \frac{-L_{n-3}(\frac{b-1}{b+1})}{(b-1)^{n-2}}+  \frac{\log(2)}{(b+1)^{n-2}}+  \lim_{u\rightarrow 1-} \left( -\frac{\log(1-u)}{(b-u)^{n-2}}+ \frac{L_{n-3}(\frac{b-1}{b-u})}{(b-1)^{n-2}}\right).$$
By Lemma  \ref{L_n_formula}, we have the limit
$$ \lim_{u\rightarrow 1-} \left( -\frac{\log(1-u)}{(b-u)^{n-2}}+ \frac{L_{n-3}(\frac{b-1}{b-u})}{(b-1)^{n-2}}\right) = 
\frac{P_{n-3}(1) -\log(b-1)}{(b-1)^{n-2}}.$$
Combining, we get
$$(n-1)(n-2)J_2(b) = \left(\frac{-\log(b-1)+ P_{n-3}(1)-L_{n-3}(\frac{b-1}{b+1})}{(b-1)^{n-2}}\right)+ \left(\frac{\log(2)}{(b+1)^{n-2}}\right).$$
Thus
$$(n-1)(n-2)M_n(b) \geq  \frac{\log\left(\frac{b+1}{2b}\right)+P_{n-2}(1)+P_{n-3}(1) -L_{n-3}(\frac{b-1}{b+1})}{(b-1)^{n-2}} 
+ \frac{\log\left(\frac{4b}{b^2-1}\right)-P_{n-2}(1)}{(b+1)^{n-2}}.$$
For $b \in (1,2]$, we have
$$\log\left(\frac{b+1}{2b}\right) + P_{n-2}(1) \geq \log\left(\frac{3}{4}\right) +1 > 0 \quad \text{and} \quad -L_{n-3}\left(\frac{b-1}{b+1}\right) > 0,$$
giving
$$(n-1)(n-2)M_n(b) \geq  \frac{P_{n-3}(1)}{(b-1)^{n-2}} 
+ \frac{\log\left(\frac{b+1}{2b}\right) + P_{n-2}(1)}{(b-1)^{n-2}}+ \frac{\log\left(\frac{4b}{b^2-1}\right)-P_{n-2}(1)}{(b+1)^{n-2}}.$$
As $(b+1)/(b-1) \geq 3$ on $(1,2]$, we have

$$\frac{\log\left(\frac{b+1}{2b}\right) + P_{n-2}(1)}{(b-1)^{n-2}} = \left(1-\frac{1}{3^{n-2}}\right)\left(\frac{\log\left(\frac{b+1}{2b}\right) + P_{n-2}(1)}{(b-1)^{n-2}} \right) + \frac{1}{3^{n-2}}\left(\frac{\log\left(\frac{b+1}{2b}\right) + P_{n-2}(1)}{(b-1)^{n-2}}\right)$$
$$\geq \left(1-\frac{1}{3^{n-2}}\right)\left(\frac{\log\left(\frac{b+1}{2b}\right) + P_{n-2}(1)}{(b-1)^{n-2}}\right) + \frac{\log\left(\frac{b+1}{2b}\right) + P_{n-2}(1)}{(b+1)^{n-2}}.$$
Therefore,
$$(n-1)(n-2)M_n(b) \geq  \frac{P_{n-3}(1)+\left(1-\frac{1}{3^{n-2}}\right)\left(\log\left(\frac{b+1}{2b}\right) + P_{n-2}(1)\right)}{(b-1)^{n-2}} $$
$$+ \frac{\log\left(\frac{b+1}{2b}\right) + P_{n-2}(1)+ \log\left(\frac{4b}{b^2-1}\right)-P_{n-2}(1)}{(b+1)^{n-2}}.$$
This gives
$$ (n-1)(n-2)M_n(b) \geq  \frac{P_{n-3}(1)+\left(1-\frac{1}{3^{n-2}}\right)\left(\log\left(\frac{b+1}{2b}\right) + P_{n-2}(1)\right)}{(b-1)^{n-2}} + \frac{\log\left(\frac{2}{b-1}\right) }{(b+1)^{n-2}}.$$
Finally,
$$(n-1)(n-2)M_n(b) \geq   \frac{P_{n-3}(1)+\left(1-\frac{1}{3^{n-2}}\right)\left(\log\left(\frac{3}{4}\right) + P_{n-2}(1)\right)} {(b-1)^{n-2}} \geq \frac{0.474879}{(b-1)^{n-2}}.$$
\end{proof}

With this bound in hand, we now find a lower bound for $F_n(x)$ by integration.

\begin{lemma}
For  $l \leq \frac{1}{2} \log\left(\frac{5}{2}\right) $, we have 
$$F_n(l) \geq \frac{K_n}{(e^{l}-1)^{n-2}},$$
where 
$$K_n = \frac{(P_{n-3}(1)+\left(1-\frac{1}{3^{n-2}}\right)\left(P_{n-2}(1)+\log\left(\frac{3}{4}\right)\right))2^{n-2}V_{n-2}V_{n-3}\Gamma(\frac{n}{2})^2}{(n-2)^2V_{n-1}\Gamma(n)}.$$
\label{Fb}
\end{lemma}

\begin{proof}
We let $a = e^l$. Then by Lemma \ref{Munif} above we have 
$$ M_n\left(\sqrt{\frac{a^2-r^2}{1-r^2}}\right) \geq \frac{A_n}{\left(\sqrt{\frac{a^2-r^2}{1-r^2}}-1\right)^{n-2}}  \qquad \mbox{for } \left(\sqrt{\frac{a^2-r^2}{1-r^2}}\right) \leq 2,$$
where 
$$A_n = \frac{P_{n-3}(1)+\left(1-\frac{1}{3^{n-2}}\right)\left(P_{n-2}(1)+\log\left(3/4\right)\right)}{(n-1)(n-2)}.$$ 
Solving this, for $r < \sqrt{(4-a^2)/3} = r(a)$ we obtain
$$F_n(l) \geq \frac{2^{n-1}V_{n-2}V_{n-3}}{V_{n-1}}\int_0^{r(a)} \frac{r^{n-3}}{\left(\sqrt{1-r^2}\right)^{n-2}}.\frac{A_n}{\left(\sqrt{\frac{a^2-r^2}{1-r^2}}-1\right)^{n-2}}
dr.$$
Simplifying we get
$$F_n(l) \geq \frac{2^{n-1}A_nV_{n-2}V_{n-3}}{V_{n-1}}\int_0^{r(a)} \frac{r^{n-3}}{\left(\sqrt{a^2-r^2}-\sqrt{1-r^2}\right)^{n-2}}
dr.$$
$$= \frac{2^{n-1}A_nV_{n-2}V_{n-3}}{V_{n-1}}\int_0^{r(a)} r^{n-3}\left(\frac{\sqrt{a^2-r^2}+\sqrt{1-r^2}}{a^2-1}\right)^{n-2}
dr.$$
As $\sqrt{a^2-r^2}/\sqrt{1-r^2} \geq a$, then  $\sqrt{a^2-r^2}+\sqrt{1-r^2}\geq (a+1)\sqrt{1-r^2}$, giving
$$F_n(l) \geq \frac{2^{n-1}A_nV_{n-2}V_{n-3}}{(a-1)^{n-2}V_{n-1}}\int_0^{r(a)} r^{n-3}\left(\sqrt{1-r^2}\right)^{n-2}
dr.$$
Therefore,
$$F_n(l) \geq \frac{2^{n-1}A_nV_{n-2}V_{n-3}}{(a-1)^{n-2}V_{n-1}}\int_0^{r(a)} r^{n-3}(1-r^2)^{n/2-1}
dr.$$
We change the variable to $t = r^2$ to get
$$F_n(l) \geq \frac{2^{n-2}A_nV_{n-2}V_{n-3}}{(a-1)^{n-2}V_{n-1}}\int_0^{r(a)^2} t^{n/2-2}(1-t)^{n/2-1}dt.$$
The Beta function $B(a,b)$ and the incomplete Beta function $B(x : a,b)$ are defined by
$$B(a,b) = \int_0^{1} t^{a-1}(1-t)^{b-1}dt\qquad B(x: a,b) = \int_0^{x} t^{a-1}(1-t)^{b-1}dt.$$
Therefore,
 $$ F_n(l) \geq \frac{2^{n-2}A_nV_{n-2}V_{n-3}}{(a-1)^{n-2}V_{n-1}} B\left(r(a)^2:\frac n 2-1,\frac n 2\right).$$
We note that  
$$B(a-1,a) = B(1/2: a-1,a) +B(1/2: a, a-1).$$ 
On $[0,1/2]$, as $t < 1-t$, we have $ t^{a-1}(1-t)^{a-2} \leq t^{a-2}(1-t)^{a-1}$ giving
 $$B(1/2: a-1,a) \geq B(1/2: a, a-1).$$ 
 Thus 
 $B(1/2: a-1,a) \geq B(a-1,a)/2$.

Therefore, if we let $r(a)^2 \geq 1/2$, then 
$$B\left(r(a)^2 : \frac n 2-1,\frac n 2\right) \geq  \frac{1}{2} B\left(\frac{n}{2}-1,\frac{n}{2}\right) = \frac{\Gamma(n/2-1)\Gamma(n/2)}{2\Gamma(n-1)} =\left(\frac{n-1}{n-2}\right) \frac{\Gamma(n/2)^2}{\Gamma(n)}  .$$
For $r(a)^2 \geq  1/2$ we require $a \leq \sqrt{5/2}$. Therefore, for $l \leq \frac{1}{2} \log\left(\frac{5}{2}\right)$ we have
$$ F(l) \geq  \left(\frac{\left(P_{n-3}(1)+\left(1-\frac{1}{3^{n-2}}\right)\left(P_{n-2}(1)+\log\left(\frac{3}{4}\right)\right) \right) 2^{n-2}V_{n-2}V_{n-3}\Gamma(\frac{n}{2})^2}{(n-2)^2V_{n-1}\Gamma(n)}\right)\frac{1}{(e^l-1)^{n-2}}.$$
\end{proof}

\section{Systole and volume estimates}
We now use the bound for $F(l)$ to obtain a lower bound on the length of the shortest orthogeodesic and to obtain lower bounds on volume in terms of the area of the boundary. We first will need the following elementary calculation.

\begin{lemma}
The constants $K_n$ from Lemma~\ref{Fb} satisfy
 $$K_n \geq 
\left(\frac{2\pi e}{n-1}\right)^{\frac{n-1}{2}}\left(\frac{3(P_{n-3}(1)+\left(1-\frac{1}{3^{n-2}}\right)\left(P_{n-2}(1)+\log\left(\frac{3}{4}\right)\right))}{2^{\frac{3}{2}}e^{\frac52}(n-2)} \right).$$
\label{Kn}
\end{lemma}

\begin{proof}
The volumes of spheres are given by 
$$V_n = \frac{(n+1)\pi^{\frac{n+1}{2}}}{\Gamma(\frac{n+3}{2})}.$$
We have Legendre's replacement formula 
$$\Gamma(z)\Gamma(z+1/2) = 2^{1-2z}\sqrt{\pi}\Gamma(2z).$$
Thus
$$\frac{2^{n-2} V_{n-2}V_{n-3}\Gamma(\frac{n}{2})^2}{(n-2)^2V_{n-1}\Gamma(n)}  = \frac{(n-1)2^{n-2}\pi^{\frac{n-3}{2}}\Gamma(\frac{n+2}{2})\Gamma(\frac{n}{2})}{(n-2)n.\Gamma(\frac{n+1}{2})\Gamma(n)}  =\frac{(n-1)\pi^{\frac{n-2}{2}}\Gamma(\frac{n+2}{2})}{2(n-2)n\Gamma(\frac{n+1}{2})^2} .$$
By using the upper and lower bounds  for the Gamma function
$$\sqrt{2\pi} x^{x+1/2}e^{-x} \leq\Gamma(x+1)  \leq e.x^{x+1/2}e^{-x},$$
we obtain
$$\frac{(n-1)\pi^{\frac{n-2}{2}}\Gamma(\frac{n+2}{2})}{2(n-2)n\Gamma(\frac{n+1}{2})^2}  \geq \frac{(n-1)\pi^{\frac{n-2}{2}}(\sqrt{2\pi}\left(\frac{n}{2}\right)^{\frac{n+1}{2}}e^{-\frac{n}{2}})}{2(n-2)n(e^2\left(\frac{n-1}{2}\right)^ne^{-(n-1)})} = \frac{2^{\frac{n}{2}-1}\pi^{\frac{n-1}{2}}n^{\frac{n-1}{2}}e^{\frac{n}{2}-3}}{(n-2)\left(n-1\right)^{n-1}}.$$
Thus
$$K_n \geq 
 \left(\frac{2\pi n e}{(n-1)^2}\right)^{\frac{n-1}{2}}\left(\frac{P_{n-3}(1)+\left(1-\frac{1}{3^{n-2}}\right)\left(P_{n-2}(1)+\log\left(\frac{3}{4}\right)\right)}{(n-2)e^{\frac52}\sqrt{2}} \right).$$
Finally as $(\frac{n}{n-1})^{\frac{n-1}{2}}$ is monotonically increasing, then as $n \geq 3$ we have  $(\frac{n}{n-1})^{\frac{n-1}{2}} \geq \frac{3}{2}$
 $$K_n \geq 
\left(\frac{2\pi e}{n-1}\right)^{\frac{n-1}{2}}\left(\frac{3(P_{n-3}(1)+\left(1-\frac{1}{3^{n-2}}\right)\left(P_{n-2}(1)+\log\left(\frac{3}{4}\right)\right))}{2(n-2)e^{\frac52}\sqrt{2}} \right).$$
\end{proof}

We now can prove the bound in Theorem~\ref{thm} which we restate below.

{
\renewcommand{\thetheorem}{\ref{thm}}
\begin{theorem}
Let $M$ be a compact hyperbolic $n$-manifold  with totally geodesic boundary. Then either  $\orth(M) > \frac{1}{2}\log\left(\frac{5}{2}\right)$ or satisfies
$$e^{\orth(M)}-1 \geq g_n \sqrt{\frac{2\pi e}{n-1}}(\Vol(M))^{-\frac{1}{n-2}},$$
where $g_n$ is an explicit monotonically increasing function tending to $1$. 
\end{theorem}
\addtocounter{theorem}{-1}
}

\begin{proof} 
Let $L =   \orth(M)$. If $L \leq  \frac{1}{2}\log\left(\frac{5}{2}\right)$, then by Lemma \ref{Fb},
$$\Vol(M) \geq F_n(L) \geq \frac{K_n}{(e^L-1)^{n-2}}.$$
Solving the latter we have
$$e^{L}-1 \geq \left(\frac{K_n}{\Vol(M)}\right)^{\frac{1}{n-2}},$$
which gives
$$e^{L}-1  \geq  K_n^{\frac{1}{n-2}} \Vol(M)^{-\frac{1}{n-2}}.$$
Therefore by Lemma \ref{Kn}, we have
$$e^{L}-1  \geq  g_n \sqrt{\frac{2\pi e}{n-1}}(\Vol(M))^{-\frac{1}{n-2}},$$
where
\begin{equation}\label{eq g_n}
g_n = \left(\frac{3\sqrt{\pi}(P_{n-3}(1)+\left(1-\frac{1}{3^{n-2}}\right)\left(P_{n-2}(1)+\log\left(\frac{3}{4}\right)\right))}{2(n-2)(n-1)^{\frac{1}{2}}e^2}\right)^{\frac{1}{n-2}}.
\end{equation}
\end{proof}

We now obtain a lower bound on the volume in terms of the boundary area. We will need an auxiliary function $S_n$ given by $$S_n(x) =\int_0^x \cosh^{n-1}(r)dr.$$

We prove Theorem~\ref{thm2}, which we first restate.

{
\renewcommand{\thetheorem}{\ref{thm2}}
\begin{theorem}
Let $M$ be a hyperbolic manifold with totally geodesic boundary. Then either
$$\Vol(M) \geq  \frac{1}{4}\log\left(\frac{5}{2}\right)\Vol(\partial M)$$ 
or 
$$ \Vol(M) \geq \frac{h_n}{3}\sqrt{\frac{2\pi e}{n-1}} \Vol(\partial M)^{\frac{n-2}{n-1}}.$$
where $h_n$ is an explicit monotonically increasing function tending to $1$.
\end{theorem}
\addtocounter{theorem}{-1}
}

\begin{proof}
Let $L=\orth(M), V = \Vol(M), A=\Vol(\partial M)$. Then by Theorem \ref{BKid},
$$V \geq F_n(L).$$
Further, the totally geodesic boundary $\partial M$ has embedded collar of radius $L/2$. By elementary hyperbolic geometry this embedded collar has volume $A.S_n(L/2)$. Thus
$$V \geq A\cdot S_n\left(\frac{L}{2}\right) \geq A \frac{L}{2}.$$
It follows that $$V \geq \max\left(F_n(L), A\frac{L}{2}\right).$$
As $F_n(x)$ is monotonically decreasing and $Ax/2$ monotonically increasing we have a unique $l   > 0$ satisfying
$$F_n(l) = A\frac{l}{2}.$$
Furthermore, it follows that $V \geq A\frac{l}{2}.$
If $l \geq \frac{1}{2}\log\left(\frac{5}{2}\right)$, then
$$V \geq \frac{1}{4}\log\left(\frac{5}{2}\right) A$$
giving the first inequality of the theorem. 

Now assume that $l \leq \frac{1}{2}\log\left(\frac{5}{2}\right)$. Then by Lemma \ref{Fb}, 
$$V \geq \max\left(\frac{K_n}{(e^l-1)^{n-2}}, A\frac{l}{2}\right).$$
We therefore consider $l_0$, the unique solution of 
$$\frac{K_n}{(e^{l_0}-1)^{n-2}} = A\frac{l_0}{2}.$$
We observe that $l_0 \leq l$ and therefore
we have $l_0 \leq \frac{1}{2}\log\left(\frac{5}{2}\right)$. 
Solving
$$A \frac{l_0}{2} = \frac{K_n}{(e^{l_0}-1)^{n-2}}$$
we obtain
$$(e^{l_0}-1)l_0^{\frac{1}{n-2}} = \left(\frac{2K_n}{A}\right)^{\frac{1}{n-2}}.$$
Thus as $l_0 <  \frac{1}{2}\log\left(\frac{5}{2}\right)$  and $(e^x-1)/x$ is monotonically increasing, we have $e^{l_0}-1 \leq a l_0$ where
 $$a= \frac{\sqrt{5/2}-1}{\log(\sqrt{5/2})} = 1.26846.$$ 
Hence we have
$$a.l_0^{\frac{n-1}{n-2}} \geq \left(\frac{2K_n}{A}\right)^{\frac{1}{n-2}},$$
$$l_0 \geq \frac{1}{a^{\frac{n-2}{n-1}}} \left(\frac{2K_n}{A}\right)^{\frac{1}{n-1}} \geq \frac{1}{a} \left(\frac{K_n}{A}\right)^{\frac{1}{n-1}}.$$
Combining with the inequality for $V$ we get
$$V \geq A\frac{l}{2} \geq A \frac{l_0}{2} \geq \frac{1}{2a}K_n^{\frac{1}{n-1}}A^{\frac{n-2}{n-1}}.$$
Hence by Lemma \ref{Kn} above 
$$V \geq \frac{h_n}{2a}\sqrt{\frac{2\pi e}{n-1}}A^{\frac{n-2}{n-1}},$$
where
\begin{equation}\label{eq h_n}
h_n = \left(\frac{3(P_{n-3}(1)+\left(1-\frac{1}{3^{n-2}}\right)\left(P_{n-2}(1)+\log\left(\frac{3}{4}\right)\right))}{2^{3/2}e^{5/2}(n-2)}\right)^{\frac{1}{n-1}}.
\end{equation}
For $n \geq 3$ it is easy to check that $h_n$ is monotonically increasing to $1$.  Evaluating $a$ we get
$$V \geq \frac{h_n}{2.53692}\sqrt{\frac{2\pi e}{n-1}}A^{\frac{n-2}{n-1}}\geq  \frac{h_n}{3}\sqrt{\frac{2\pi e}{n-1}}A^{\frac{n-2}{n-1}}.$$
\end{proof}

\section{Dimension 3 case}\label{sec5}

 We note that the constants in the main theorems proved for general dimension can be improved in any specific case by analysing $F_n$ individually. We now consider the three dimensional case separately. 
 
 In \cite{MMc}, Masai--McShane proved that the volume identity of Bridgeman--Kahn (see Theorem~\ref{BKid}) is equal to the identity obtained by Calegari (see \cite{DCid}) using a different decomposition. Applying Calegari's formula in dimension three they obtained an elementary closed form for $F_3$ which gives
\begin{equation}
\label{3d}
F_3(x) = 2\pi \left(\frac{x+1}{e^{2x}-1} \right).
\end{equation}
We note that there is a normalization error in \cite{MMc} (by a factor of $4\pi$) and the above formula is the corrected version (see \cite{VY} where the correct version is also stated).

Using the formula of Masai--McShane for $F_3$ we  can give an elementary argument that improves the constants in Theorem \ref{thm} in the case of $n=3$. We would like to thank the referee for this observation.

\begin{prop}
Let $M$ be a compact hyperbolic $3$-manifold with non-empty totally geodesic boundary.  Then either $\orth(M) > 1.25$ or
$$e^{\orth(M)}-1\geq  \frac{\pi}{V(M)}.$$
\end{prop}

\begin{proof}
By elementary calculus for $0\leq x \leq 1.25$ we have 
$$\frac{x+1}{e^x+1} \geq \frac{1}{2}.$$
Thus for $\orth(M) \leq 1.25$, equation \eqref{3d} gives
$$V(M) \geq F(\orth(M)) = 2\pi\left(\frac{\orth(M)+1}{e^{\orth(M)}+1}\right)\frac{1}{e^{\orth(M)}-1} \geq \frac{\pi}{e^{\orth(M)}-1}.$$
Thus, if $\orth(M) \leq 1.25$,
$$e^{\orth(M)}-1\geq  \frac{\pi}{V(M)}.$$
\end{proof}

We now compare this with Theorem \ref{thm}. For $n=3$  the theorem states  
that if $\orth(M) \leq \frac{1}{2}\log(5/2)$, then
 $$e^{\orth(M)}-1 \geq \frac{g_3 \sqrt{\pi e}}{V(M)}= \frac{0.353076}{V(M)}.$$

Also in dimension $3$, Miyamoto and Kojima proved that Miyamoto's bound in \cite{Miy94} is optimal and that the lowest volume hyperbolic $3$-manifold with totally geodesic boundary has boundary a genus two surface and volume $6.452$ (see \cite{KM91}). We can compare this optimal bound to the bound obtained using equation \eqref{3d} for $F_3$.  

As in our prior analysis in Theorem \ref{thm2}, we obtain a volume bound by finding the common value of $F_3(x) =  4\pi S_3(x/2)$. Solving numerically we obtain a lower bound of $4.079$ which is comparable to Miyamoto's optimal bound. This was also observed in \cite[Section 7]{BK10} but due to the missing factor in the integral formula for $F_n$ (see the footnote attached to equation \eqref{Fn}), the bound obtained there was given as $2.986$.

\section{Lower bounds for volume of hyperbolic $n$-manifolds with totally geodesic boundary}\label{sec6}

We now consider our bounds in general dimension $n \geq 3$. In even dimensions the generalized Gauss--Bonnet theorem gives
$$\Vol(M) = \frac{|\chi(M)|V_{n}}{2} \geq \frac{V_n}{2} = \frac{(n+1)\pi^{\frac{n+1}{2}}}{\Gamma(\frac{n+3}{2})}.$$
For odd dimensions the best lower bound is by Adeboye and Wei (see \cite{AW}) with
\begin{equation}\label{eq-AW}
\Vol(M) \gtrsim \left(\frac{2}{n}\right)^{\frac{n^2}{2}}.
\end{equation}
Miyamoto \cite{Miy94} showed that for a hyperbolic manifold $M$ with non-empty totally geodesic boundary we have 
$$\Vol(M) \geq \rho_n \Vol(\partial M)$$
for some constants $\rho_n$. In \cite[Lemma~1.4.3 and Table~1.4.5]{KelHab} (see also \cite{Kel95}), Kellerhals showed that $\rho_n$ are monotonically increasing with $\rho_6 = 0.64652$. Thus for $n > 6$ odd we have
 $$\Vol(M) \geq  \frac{\rho_n}{2} \frac{V_{n-1}}{2} \geq 0.32326 V_{n-1}.$$
When applies, this bound is much stronger than \eqref{eq-AW} (applied to the double of $M$).

The key ingredient of Miyamoto's proof is his notion of the \emph{hypersphere packings}. These packings have similar properties to the sphere packings in constant curvature spaces. In his paper Miyamoto proved a hypersphere analogue of the well known B\"or\"oczky's sphere packing theorem which says that any sphere packing of radius $r$ in an $n$-dimensional space of constant curvature has density at most that of $n+1$ mutually touching balls in the regular $n$-simplex of edgelength $2r$ spanned by their centers. Following this line of argument, the constant $\rho_n$ in Miyamoto's volume bound is given by the ratio of the volumes of a certain truncated and regular hyperbolic simplices. These volumes can be further related to the volumes of orthoschemes. In her thesis \cite{KelHab}, Kellerhals was able to explicitly estimate the latter volumes.   

We now show that our results give a new proof of a linear bound for $\Vol(M)$. By Theorem~\ref{thm2}
either
$$\Vol(M) \geq \frac{1}{4}\log\left(\frac{5}{2}\right)\Vol(\partial M)$$
or
$$\Vol(M) \geq \frac{h_n}{3}\sqrt{\frac{2\pi e}{n-1}} (\Vol(\partial M))^{\frac{n-2}{n-1}},$$
where $h_n$ monotonically increases to $1$. The first bound is linear and implies for $n$ odd
$$\Vol(M) \geq \frac{1}{8}\log\left(\frac{5}{2}\right)V_{n-1}.$$
To show that the second bound also gives us a linear lower bound in terms of $V_{n-1}$ we note that by Stirling's approximation
$$V_n= \frac{(n+1)\pi^{\frac{n+1}{2}}}{\Gamma(\frac{n+3}{2})} \leq \frac{1}{\sqrt{2}}\left(\frac{2\pi e}{n+1}\right)^{\frac{n}{2}} \leq  \frac{1}{\sqrt{2}}\left(\frac{2\pi e}{n}\right)^{\frac{n}{2}}.$$

Therefore, 
$$\Vol(M) \geq \frac{h_n}{3}\sqrt{\frac{2\pi e}{n-1}}\left(\frac{V_{n-1}}{2}\right)^{\frac{n-2}{n-1}}  \geq \frac{ h_n}{3} \frac{V_{n-1}}{2} = \frac{h_n}{6}V_{n-1}.$$
Thus for $n$ odd we have
$$\Vol(M) \geq \min\left(\frac{1}{8}\log\left(\frac{5}{2}\right),\frac{h_n}{6}\right) V_{n-1}$$
proving  Theorem~\ref{thm-vol}.

This way we obtain another proof of a lower bound linear in $V_{n-1}$ using different methods. The answers are remarkably similar in spite of the different approaches. To compare, our method  gives a  linear constant tending to  $\frac{1}{8}\log\left(\frac{5}{2}\right) \simeq 0.11453$ and Miyamoto--Kellerhals give a slightly better bound of $0.32326$. It would be interesting to see if there is any deeper relation between the two.

In conclusion let us remark that it is widely believed that these bounds for volumes of hyperbolic manifolds, as well as the Gauss--Bonnet bound in even dimensions, are far from sharp. The sharp bounds are known for \emph{arithmetic} orbifolds, and they imply good bounds for arithmetic manifolds (see \cite{Bel14, BE14}). These bounds grow superexponentially fast with the dimension. It is not known if there exists a hyperbolic $n$-manifold whose volume is less than the minimal volume of an arithmetic $n$-manifold.


\medskip


\begin{thebibliography}{1}
\bibitem{AW} I. Adeboye and G Wei, 
\newblock On volumes of hyperbolic orbifolds, 
\newblock {\em Algebr. Geom. Topol.}, 12 (2012), 215--233.

\bibitem{Bas93}
A. Basmajian,
\newblock The orthogonal spectrum of a hyperbolic manifold,
\newblock {\em Amer. J. Math.}, 115 (1993), 1139--1159.

\bibitem{Bel14}
M. Belolipetsky, 
\newblock Hyperbolic orbifolds of small volume, 
\newblock Proceedings of the International Congress of Mathematicians -- Seoul 2014. Vol. II, 837--851, Kyung Moon Sa, Seoul, 2014.

\bibitem{BE14}
M. Belolipetsky and V. Emery, 
\newblock Hyperbolic manifolds of small volume, 
\newblock {\em Doc. Math.}, 19 (2014), 801--814.

\bibitem{BT11} M. Belolipetsky and S. A. Thomson, 
\newblock Systoles of hyperbolic manifolds, 
\newblock {\em Algebr. Geom. Topol.}, 11 (2011), 1455--1469. 


\bibitem{BK10}
M. Bridgeman and J. Kahn,
\newblock Hyperbolic volume of $n$-manifolds with geodesic boundary and orthospectra,
\newblock {\em Geom. Funct. Anal.}, 20 (2010), 1210--1230.

\bibitem{DCid} 
D. Calegari, 
\newblock Chimneys, leopard spots and the identities of Basmajian and Bridgeman. 
\newblock {\em Algebraic \& Geometric Topology}, 10(3), 2010, 1857-1863.


\bibitem{KelHab} R. Kellerhals, 
\newblock Volumina von hyperbolischen Raumformen.
\newblock Habilitationsschrift, Universit\"at Bonn 1995, Preprint Max-Planck-Institut f\"ur Mathematik Bonn MPI 95-110.

\bibitem{Kel95} R. Kellerhals, 
Regular simplices and lower volume bounds for hyperbolic $n$-manifolds,
\newblock {\em Ann. Global Anal. Geom.}, 13 (1995), 377--392.

 \bibitem{MMc} H. Masai and G. McShane, 
 \newblock Equidecomposability,volume formulae and orthospectra
 \newblock {\em Algebraic \& Geometric Topology}, 13, 3135--3152, (2013).

 \bibitem{KM91} S. Kojima and Y. Miyamoto, 
 \newblock The smallest hyperbolic 3-manifolds with totally geodesic boundary.
 \newblock {\em J. Differential Geometry}, 34, 175--192, (1991).


\bibitem{Miy94} Y. Miyamoto, \newblock Volumes of hyperbolic manifolds with geodesic boundary, 
\newblock {\em Topology}, 33 (1994), 613--629. 


\bibitem{VY} N. Vlamis and A. Yarmola,
\newblock The Bridgeman--Kahn identity for hyperbolic manifolds with cusped boundary,
\newblock {\em Geometriae Dedicata}, 194 (2018), 81--97.

\end{thebibliography}
\end{document}